\documentclass[
11pt,  reqno]{amsart}
\usepackage[margin=1.1in,marginparwidth=1.5cm, marginparsep=0.5cm]{geometry}

\usepackage{booktabs} 
\usepackage{microtype}
\usepackage{amssymb}
\usepackage{mathrsfs}

\usepackage{upgreek} 

\usepackage{color}
\usepackage[implicit=true]{hyperref}

\usepackage{xsavebox}

\usepackage{cases}

\allowdisplaybreaks[2]

\definecolor{gr}{rgb}   {0.,   0.69,   0.23 }
\definecolor{bl}{rgb}   {0.,   0.5,   1. }
\definecolor{mg}{rgb}   {0.85,  0.,    0.85}
\definecolor{yl}{rgb}   {0.8,  0.7,   0.}
\definecolor{or}{rgb}  {0.7,0.2,0.2}

\usepackage{tikz} 

\usepackage{marginnote}
\usepackage{scalerel} 

\usetikzlibrary{shapes.misc}
\usetikzlibrary{shapes.symbols}
\usetikzlibrary{decorations}
\usetikzlibrary{decorations.markings}

\newtheorem{theorem}{Theorem} [section]

\newtheorem{lemma}[theorem]{Lemma}
\newtheorem{proposition}[theorem]{Proposition}
\newtheorem{remark}[theorem]{Remark}

\newtheorem{corollary}[theorem]{Corollary}

\newcommand{\noi}{\noindent}
\newcommand{\Z}{\mathbb{Z}}
\newcommand{\R}{\mathbb{R}}
\newcommand{\T}{\mathbb{T}}

\let\P= \undefined
\newcommand{\P}{\mathbf{P}}

\newcommand{\E}{\mathbb{E}}

\newcommand{\F}{\mathcal{F}}

\newcommand{\dl}{\delta}

\newcommand{\eps}{\varepsilon}

\newcommand{\ld}{\lambda}
\newcommand{\Ld}{\Lambda}
\newcommand{\s}{\sigma}

\newcommand{\ft}{\widehat}
 
\newcommand{\wt}{\widetilde}
\newcommand{\cj}{\overline}

\newcommand{\embeds}{\hookrightarrow}

\renewcommand{\l}{\ell}
\renewcommand{\o}{\omega}
\renewcommand{\O}{\Omega}

\newcommand{\les}{\lesssim}

\newcommand{\jb}[1]
{\langle #1 \rangle}

\newcommand{\pa}{\partial}

\newcommand{\B}{\mathcal{B}}

\newcommand{\FL}{\mathcal{F} L}

\newcommand{\C}{\mathcal{C}}
\numberwithin{equation}{section}
\numberwithin{theorem}{section}

\newcommand{\PP}{\mathbb{P}}

\DeclareMathOperator{\Law}{Law}

\newcommand{\D}{\mathcal{D}}

\usepackage{comment}

\begin{document}

\title[LDP for NLS]
{Large deviations principle for the cubic NLS equation with slowly decaying data}

\author[R. Liang, Y. Wang]
{Rui Liang and Yuzhao Wang}

\address{
Rui Liang\\
School of Mathematical Sciences\\
South China Normal University\\
Guangzhou\\
Guangdong\\
510631\\
P. R. China\\
and
School of Mathematics and Statistics\\
University of Massachusetts Amherst\\
Amherst\\
Massachusetts\\
MA 01003\\
USA\\
and 
School of Mathematics\\
Watson Building\\
University of Birmingham\\
Edgbaston\\
Birmingham\\
B15 2TT\\ United Kingdom}
\email{RXL833@alumni.bham.ac.uk}

\address{
Yuzhao Wang\\
School of Mathematics\\
Watson Building\\
University of Birmingham\\
Edgbaston\\
Birmingham\\
B15 2TT\\ United Kingdom}
\email{y.wang.14@bham.ac.uk}



\begin{abstract}
In this note, 
we prove a sharp large derivation principle (LDP) for the cubic nonlinear Schr\"odinger equation with Gaussian random initial data in Fourier Lebesgue spaces.
As a consequence, we improve the exponential decay condition in \cite{GGKS21} to $\l^1$ decay.

\end{abstract}


\maketitle


\section{Introduction}

In this paper, we consider the following Cauchy problem of the one-dimensional cubic nonlinear Schr\"odinger equation (NLS) on the torus $\T= \R / (2 \pi \Z)$: 
\begin{equation}
\label{NLS}
\begin{cases}
i\pa_t u + \partial_x^2 u = \pm \,|u|^2\, u,\\
u(t,x)|_{t=0}=u_0.
\end{cases}
\end{equation}

\noi 
The equation \eqref{NLS} is fundamental in theoretical physics and applied mathematics, for describing the evolution of wave packets in various physical systems, such as nonlinear optics, fluids, and plasmas; see \cite{SS99} for a review.
Bourgain \cite{BO93} proved \eqref{NLS} is deterministically global well-posed in $L^2 (\T)$.
The probabilistic study of \eqref{NLS} has emerging importance since the seminal work of Bourgain on invariant measures \cite{BO94,BO96}.

More recently, probabilistic analyses of \eqref{NLS} have also found applications in oceanography, particularly in modelling rare extreme events such as the formation of {\it rogue waves} in deep-sea dynamics~\cite{DGV18,DGV19,GGKS21}. In particular, Garrido et al.~\cite{GGKS21} established a large deviations principle (LDP) describing the probability of observing a wave of unusually large height in the weakly nonlinear regime. Their result relied on a strong exponential decay assumption on the Fourier coefficients of the random initial data, which plays a crucial role in their analysis.
In this paper, we improve the exponential decay condition on coefficients of initial data in \cite{GGKS21} to a $\l^1$ decay condition.

We consider 
random initial data of the form:
\begin{equation}\label{random_data}
u_0 (x)= \sum_{k\in\Z} c_k g_k e^{ikx},
\end{equation}
where $c_k \in \C$ and $\{g_k\}_{k\in\Z}$ are independent, identically distributed, complex Gaussian random variables with $\E g_k=0$, $\E g_k g_j=0$ and $\E g_k \overline{g}_j =\delta_{kj}$ for $k,j\in\Z$.  
Without loss of any generality, we only consider nonnegative real coefficients $c_k \ge 0$, due to the rotation invariance of complex Gaussian random variables.

The study of the large deviation principle for solutions to \eqref{NLS} was initiated by \cite{DGV18,DGV19}, where the authors conjectured that the existence of an LDP can be used to predict the formation of rogue waves.
In particular, they considered \eqref{NLS} with random initial data $u_0^N$ of the form
\begin{align}
\label{random1}
u_0^N (x) = \sum_{|k| \le N} \theta_k e^{ikx},
\end{align}

\noi
where the initial data is parametrized by a random vector $\theta = (\theta_k)_{|k| \le N} \in \mathbb C^{2N+1}$.
Here $(\theta_k)$ are independent complex Gaussian random variables with $\E \theta_k = \E \theta_k^2 = 0$ and $\E |\theta_k|^2 = c_k^2$, for some fast-decaying $c_k > 0$.
Then, the set of initial data that generates a rouge wave of height at least $z>0$ at time $t>0$ is given by
\begin{align}
\label{Dtz}
\mathcal D (t,z) : = \bigg\{ (\theta_k)_{|k| \le N} \in \C^{2N+1} \Big| \sup_x |u (t,x| \theta)| > z\bigg\},
\end{align}

\noi 
where $u$ is the solution to NLS \eqref{NLS} with the truncated initial data \eqref{random1}.
Dematteis et al, then proposed a theoretical framework of a large deviations principle (LDP) to quantify the likelihood of $\mathcal D (t,z)$.
In particular, consider the minimization problem
\begin{align}\label{minp}
\theta^* (z) = {\rm argmin}\,\, \mathcal I (\theta)
\end{align}

\noi 
where
\[
\mathcal I (\theta) = \max_{y \in \C^{2N+1}} [\jb{y,\theta} - S(y)] \quad \textup{ and } \quad S(y) = \log \E e^{\jb{y, \theta}}.
\]

\noi
Then Dematteis et al. claim that for $t >0$ 
\begin{align}
\label{LDP1}
\log \PP (\mathcal D(t,z)) = - \mathcal I (\theta^* (z)) + o(1),
\end{align}

\noi 
as $z \to \infty$, provided 
the minimization problem \eqref{minp} has a unique solution.

Garrido et al.~\cite{GGKS21} rigorously proved the conjectured LDP~\eqref{LDP1} locally in time for the nonlinear Schr\"odinger equation with weak nonlinearity. 
A major difficulty in analysing the LDP problem~\eqref{LDP1} lies in estimating the gradient 
$\nabla_{\theta} \sup_{x} |u(t,x|\theta)|$ 
and in verifying the convexity of the set $\mathcal{D}(t,z)$, 
both of which are extremely challenging to establish in practice. 
To overcome these obstacles, Garrido et al.~\cite{GGKS21} introduced a new formulation and considered the following LDP problem:
\begin{align}
\label{LDP2}
\log \PP \Big( \sup_{x} |u(t,x)| > z \Big) = - I(z) + o(1), \quad z \to \infty,
\end{align}
where $I$ denotes the corresponding rate function. 
The advantage of this reformulation is that it circumvents the need to solve the minimization problem in~\eqref{minp}. 
Within this framework, they established~\eqref{LDP2} for initial data of the form~\eqref{random1} with infinitely many Fourier modes ($N = \infty$) satisfying a \emph{strong decay condition} on the coefficients.

In particular, consider the weakly nonlinear Schr\"odinger equation on the circle $\T = [0,2\pi]$:
\begin{align}
\label{weakNLS1}
i\pa_t u + \Delta u = \varepsilon^{2} |u|^2 u.
\end{align}
It was proved in~\cite{GGKS21} that
\begin{align}
\label{THM_LDP3}
\lim_{\varepsilon \to 0^+} 
\varepsilon \log \PP \Big( \sup_{x} |u(t,x)| > z_0 \varepsilon^{-1/2} \Big)
= - \frac{z_0^2}{\sum_k c_k^2},
\end{align}
for $0 < t \lesssim \varepsilon^{-1}$ and $z_0 > 0$,
where $u(t,x)$ denotes the solution to~\eqref{weakNLS1} with random initial data $u_0^\omega$ given by~\eqref{random_data} and coefficients satisfying
\begin{align}\label{decay}
c_k = a e^{-b|k|} 
\quad \text{or} \quad 
c_k = a e^{-b|k|^2},
\end{align}
for some fixed $a,b > 0$. 

The proof of~\eqref{THM_LDP3} is based on the G\"artner--Ellis theorem combined with resonant approximation techniques. 
In their analysis, the rapid decay condition~\eqref{decay} plays a crucial role in both the linear and nonlinear estimates. 
They further observed that alternative families of coefficients $(c_k)$ might also yield~\eqref{LDP2}, and posed an \textit{open question} concerning the identification of the optimal decay condition under which~\eqref{LDP2} holds; see~\cite[Remark~1.3]{GGKS21}. 
One of the principal aims of the present work is to address this question.

To this purpose, we consider a LDP under Fourier-Lebesgue norm, which is stronger than the $L^\infty (\T)$ norm used in \eqref{THM_LDP3}.
Our conventions for the Fourier transform are as follows:
\[
\ft f (n) = \frac1{{2\pi}} \int_0^{2\pi} f(x) e^{ix n}dx \quad \textup{ and } \quad f(x) = \sum_{n\in \Z} \ft f(n) e^{ixn}
\]

\noi 
for functions on the circle $\T = \R/(2\pi \Z)$.
With the above notations, we define the Fourier-Lebesgue norm by
\[
\begin{split} 
\|f\|_{\FL^1 (\T)} = \sum_{n \in \Z} |\ft f(n)|.
\end{split}
\]

\noi 
It is easy to see that
\[
\sup_{x \in \T} |f(x)| \le \|f\|_{\FL^1 (\T)}.
\]

\noi 
We note that $L^\infty$ and $\FL^1$ scales the same, and furthermore, the LDPs under $L^\infty$ and $\FL^1$ share the same rate function.
In particular, for $t \les \eps^{-1}$, we have
\begin{align}\label{obser}
\lim_{\eps \to 0^+} \eps \log \PP \Big( \sup_{x} |u(t,x) | > z_0\eps^{-1/2} \Big) = \lim_{\eps \to 0^+} \eps \log \PP \Big( \|u(t,x) \|_{\FL^1} > z_0\eps^{-1/2} \Big),
\end{align}

\noi 
where $u(t,x)$ is the solution to \eqref{weakNLS1} with initial data \eqref{random1} in $\mathcal{F}L^1$.
This motivates us to consider the LDP under the $\FL^1$ norm.
We are ready to state our first main result.
\begin{theorem}
\label{THM:main} 
Let $u(t,x)$ be solutions to \eqref{weakNLS1} with initial data $u_0^\o$ given in \eqref{random_data}.
For any $t = \mathcal O ( \eps^{-1} |\log \eps|)$ and $z_0 > 0$, we have that
\begin{align}
\label{THM_LDP4}
\lim_{\eps \to 0^+} \eps \log \PP \Big( \|u(t,x) \|_{\FL^1 (\T)} > z_0\eps^{-1/2} \Big) = - \frac{z_0^2}{\sum_k c_k^2},
\end{align}

\noi 
holds if and only if $(c_k) \in \l^1$.
\end{theorem}

\begin{remark}
\rm 
We remark that the time scale $t = \mathcal O ( \eps^{-1} |\log \eps|)$ is critical;
see \cite[Remark 1.7]{GGKS21} for more discussion on critical time.
In the limit $\eps \to 0^+$, the large deviation principle \eqref{THM_LDP4} may hold for any fixed time $t > 0$.
\end{remark}

\begin{remark}
\rm 
To prove Theorem~\ref{THM:main}, we first develop a novel approach to establish
the LDP for the linear Schr\"odinger equation under the sole assumption that
the initial data lies in $\FL^1$. We then carry out the nonlinear analysis,
together with a perturbative argument, entirely within the $\FL^1$ framework.
\end{remark}

As a consequence of Theorem \ref{THM:main} and the observation \eqref{obser}, we can improve \eqref{THM_LDP3} in terms of the decay of $c_k$.

\begin{corollary}\label{COR:main} 
Let $u(t,x)$ be solutions to \eqref{weakNLS1} with initial data $u_0^\o$ given in \eqref{random_data}.
For any $t = \mathcal O ( \eps^{-1} |\log \eps|)$ and $z_0 > 0$, we have that \eqref{THM_LDP3} holds provided $(c_k) \in \l^1$.
\end{corollary} 

\begin{remark}
\rm 
We note that Corollary \ref{COR:main} improves \cite[Theorem 1.1]{GGKS21} in the following sense.
The coefficients $(c_k)$ in \cite[Theorem 1.1]{GGKS21} is required to decay condition \eqref{decay}.
In Proposition \ref{PROP:linear2}, it requires only $(c_k) \in \l^1$.
It is expected that \eqref{THM_LDP3} holds for initial data $u_0 \in L^2 \cap L^\infty$, where $\FL^1 \subset L^2 \cap L^\infty$.
However, our current argument can only handle $\FL^1$ data.
\end{remark}

\begin{remark}
\rm 
Recently, in \cite{BGMS25}, the authors studied rogue waves and large deviations
for two-dimensional pure-gravity deep-water waves, assuming initial data with
exponentially decaying Fourier coefficients; see \cite[(1.14)]{BGMS25}. It is
expected that our argument also applies in their setting and allows one to
relax this assumption to merely $\ell^1$ decay of the Fourier coefficients.
We plan to pursue this in future work.
\end{remark}


\section{LDF for linear Schr\"odinger equation}

Consider the linear Schr\"odinger equation on the torus $\T =[0,2\pi]$:
\begin{align}
\label{linear}
\begin{cases}
i \partial_t u + \partial_{xx} u = 0,\\
 u(0,x) = u_0^\o (x),
\end{cases}
\end{align}

\noi 
where the random initial data $u_0^\o$ is given by
\eqref{random_data}.
We would like to determine the optimal condition on the coefficients $(c_k)$ of \eqref{random_data} such that \eqref{THM_LDP4} holds for solutions to the linear Schr\"odinger equation \eqref{linear}.


Now, we are ready to state our first main result in this section.

\begin{proposition}
\label{PROP:linear1}
Consider the linear Schr\"odinger equation on the torus $\T = [0,2\pi]$ as in \eqref{linear}, with random initial data $u_0^\o$ given by \eqref{random_data}.
Then,
\begin{align}
\label{LDP_linear1}
\lim_{\eps\to 0^+} \eps \log \P \big( \|u(t)\|_{\FL^1(\T)} \ge z_0 \eps^{-1/2} \big) = - \frac{z_0^2}{\sum_{k \in \Z} c_k^2},
\end{align}

\noi 
holds if and only if $(c_k) \in \l^1$.
\end{proposition}




\subsection{Sharpness}
In this subsection, we prove the sharpness of the condition $(c_k) \in \l^1$ of Proposition \ref{PROP:linear1}, i.e. if $(c_k) \notin \l^1$, then \eqref{LDP_linear1} fails.
To this purpose, we prove the following result.
\begin{lemma}
\label{LEM:unbound}
Given $u_0^\o$ in \eqref{random_data} with $(c_k) \notin \l^1$, we have
\begin{align} \label{unbound0}
\| u_0^\o \|_{\FL^1 (\T)} = \infty,
\end{align}

\noi 
almost surely.
\end{lemma}

Before the proof of Lemma \ref{LEM:unbound}, we recall that the modulus of a complex Gaussian random variable with mean $0$ and variance $2 \s^2$ follows a Rayleigh distribution with parameter $\s >0$ and probability density function
\begin{align}
\label{RayDis}
f(x,\s) = \frac{x}{\s^2} e^{-\frac{x^2}{2\s^2}}, \qquad x \ge 0.
\end{align}

\noi 
Since $g_k$ is a standard complex Gaussian with mean 0 and variance $2$,
we may write $g_k = R_k e^{i \varphi_k}$,
where $\Law(R_k) = {\rm Rayleigh}(1/\sqrt 2)$
and ${\rm Law} (\varphi_k) = {\rm U} [0, 2\pi]$,
with $R_k$ and $\varphi_k$ independent of each other.
Now we are ready to prove Lemma \ref{LEM:unbound}.




\begin{proof}[Proof of Lemma \ref{LEM:unbound}]
We first claim that
\begin{align}
\label{unbound1}
\begin{split}
\E \big[ \exp (- \|u_0^\o\|_{\FL^1 (\T)}) \big] & = \E \bigg[ \exp \Big(- \sum_{k \in \Z} c_k R_k \Big) \bigg]  = \prod_{k \in \Z} \E \big[ \exp (-c_k R_k) \big] =0.
\end{split}
\end{align}

\noi 
It is easy to see that \eqref{unbound0} follows from \eqref{unbound1}.

In what follows, we prove the claim \eqref{unbound1}.
We distinguish two cases in terms of the limiting behaviour of $(c_k)$.
If one of the limits $\lim_{k \to \pm \infty} c_k$ is not zero,
then there is a number $\eps_0 > 0$ and a subsequence $(c_{k_j})$ of infinite elements such that $c_{k_j} \ge \eps_0$.
Therefore, we have 
\[
\begin{split}
\prod_{k \in \Z} \E \big[ \exp (-c_k R_k) \big] & \le \prod_{j \in \Z} \E \big[ \exp (-c_{k_j} R_{k_j}) \big] \le \prod_{j \in \Z} \E \big[ \exp (- \eps_0 R_{k_j}) \big] =0,
\end{split}
\]

\noi 
where we used the fact that $\E \big[ \exp (- \eps_0 R_{k_j}) \big] = \E \big[ \exp (- \eps_0 R_{0}) \big] < \E \big[ 1 \big] = 1$, which implies \eqref{unbound1}.

We turn to the case where both $\lim_{k \to \pm \infty} c_k = 0$.
In view of \eqref{RayDis}, we see that
\[
\begin{split}
 \E \big[ \exp (-c_k R_k) \big] & = \int_0^\infty e^{-c_k x} 2x e^{-x^2} dx  
 = 1 - c_k e^{\frac{c_k^2}4} \int_{\frac{c_k}2}^\infty e^{-x^2} dx   \le 1 - \frac{c_k}2,
\end{split}
\]

\noi
provided $c_k \ll 1$, which can be guaranteed by setting $|k| \gg 1$.
Therefore, it follows that
\[
\begin{split}
\prod_{k \in \Z} \E \big[ \exp (-c_k R_k) \big] & \le \prod_{|k| \gg 1} \bigg( 1- \frac{c_k}2\bigg) \le \exp \bigg( - \frac12\sum_{|k| \gg 1} c_k\bigg) = 0,
\end{split}
\]

\noi 
where we used the fact $(c_k)\notin \l^1$.

Thus, we finish the proof of the claim \eqref{unbound1} and conclude the proof of the lemma.
\end{proof}

Now we are ready to prove the sharpness of Proposition \ref{PROP:linear1}.

\begin{proof}[Sharpness of Proposition \ref{PROP:linear1}]
Lemma \ref{LEM:unbound} implies that 
\[
\P \big( \|u(t)\|_{\FL^1(\T)} \ge z_0 \eps^{-1/2} \big) = 1
\]

\noi 
for any $z_0, \eps \in (0, \infty)$.
Thus for any $(c_k) \notin \l^1$, it follows that
\[
\lim_{\eps\to 0^+} \eps \log \P \big( \|u(t)\|_{\FL^1(\T)} \ge z_0 \eps^{-1/2} \big) = 0.
\]

\noi 
This proves the sharpness part in Proposition \ref{PROP:linear1}.
\end{proof}






\subsection{Linear LDP I - Fourier Lebesgue norm}
In this subsection, we prove \eqref{LDP_linear1}. We rewrite
\begin{align}
\label{bound2}
\begin{split}
\PP \big( \| u(t,x) \|_{\FL^1 (\T)} \ge z_0 \eps^{-1/2} \big) = \PP \bigg( \sum_k c_k R_k \ge z_0 \eps^{-1/2} \bigg) .
\end{split}
\end{align}

\noi 
To estimate the right-hand side of \eqref{bound2}, we follow the G\"artner-Ellis theory.
Consider the continuous parameter family $\mu_\eps$ defined as the probability measures of
\begin{align}
\label{Zeps}
Z_\eps : = \eps^{1/2} \sum_k c_k R_k,
\end{align}

\noi 
for $\eps > 0$.
For $\ld \in \R$ and $\eps > 0$, we define  the {\bf cumulant-generating function} of $Z_\eps$ as
\begin{align}
\label{Geps}
\Ld_\eps (\ld) : = \log \E [e^{\ld Z_\eps}],
\end{align}

\noi 
where $Z_\eps$ has distribution $\mu_\eps$,
as in \eqref{Zeps}.

\begin{theorem}
[G\"artner-Ellis Theorem]\label{THM:GE}
If the function $\Ld (\cdot)$,
defined as the limit
\begin{align*}
\Ld (\ld) : = \lim_{\eps \to 0^+} \eps \Ld_\eps (\eps^{-1} \ld),
\end{align*}

\noi 
exists for each $\ld \in \R$,
takes values in $\R$,
and is differentiable,
then we have that
\begin{align}
\label{GE_up}
- \inf_{z \in (z_0, \infty)} \Ld^* (z) \le \liminf_{\eps \to 0^+} \eps \log \mu_\eps ((z_0, \infty)),
\end{align}

\noi 
and 
\begin{align}
\label{GE_down}
\limsup_{\eps \to 0^+} \eps \log \mu_\eps ((z_0, \infty)) \le - \inf_{z \in [z_0, \infty)} \Ld^* (z),
\end{align}

\noi 
where 
$$\Ld^* (\cdot) : = \sup_{\ld \in \R} (\ld z - \Ld (\ld))$$ 
is the Fenchel-Legendre transform of $\Ld (\cdot)$.
\end{theorem}

To apply Theorem \ref{THM:GE}, we first estimate $\Ld_\eps (\ld)$ defined in \eqref{Geps}.

\begin{lemma}
\label{LEM:Leps}
Let $(c_k) \in \l^1$ be a positive sequence.
Then we have 
\begin{equation}
\label{Leps}
\Lambda_{\varepsilon} (\varepsilon^{-1} \, \lambda) = \sum_{k \in \Z} \log \left( 1+ \sqrt{\pi}\varepsilon^{-1/2}\,\lambda\,c_k\, \exp\left(\frac{\varepsilon^{-1} \, \lambda^2 \, c_k^2}{4}\right) \, \PP (X_k\geq 0)\right),
\end{equation}

\noi 
where $X_k$ are independent normal random variables with mean $\eps^{-1/2} \ld c_k/2$ and variance $c_k^2/2$.
Furthermore, given $\ld > 0$, $\eps \Lambda_{\varepsilon} (\varepsilon^{-1} \, \lambda)$ is uniformly bounded for all $\eps \in (0,1)$.
\end{lemma}

\begin{proof}
The formula \eqref{Leps} follows from a direct computation using the density function of $R_k$ in \eqref{RayDis}.
We give a detailed computation for completeness. 
In view of the Monotone Convergence Theorem and \eqref{Zeps}, we see that
\begin{align}
\label{Lu1}
\begin{split}
\Ld_\eps (\eps^{-1} \ld) & = \log \E \big[e^{\eps^{-1/2} \ld \sum_k c_k R_k} \big] 
 = \sum_k \log \E  \big[ e^{\eps^{-1/2} \ld c_k R_k} \big],
\end{split}
\end{align}

\noi 
where we used the independence of $(R_k)$.
To get \eqref{Leps},
it suffices to show
\begin{align}
\label{Lu2}
\E  \big[ e^{\eps^{-1/2} \ld c_k R_k} \big] = 1+ \sqrt{\pi}\varepsilon^{-1/2}\,\lambda\,c_k\, \exp\left(\frac{\varepsilon^{-1} \, \lambda^2 \, c_k^2}{4}\right) \, \PP (X_k\geq 0),
\end{align}

\noi 
holds for each $k \in \Z$.
To prove \eqref{Lu2},
we see from \eqref{RayDis} that
\begin{align}
\label{Lu3}
\begin{split}
\E  \big[ e^{\eps^{-1/2} \ld c_k R_k} \big] & = \int_0^\infty e^{\eps^{-1/2} \ld c_k x} 2x e^{-x^2} dx \\
& = 1 + \sqrt{\frac{\pi}{\eps}} \ld c_k \exp \bigg( \frac{\ld^2 c_k^2}{4\eps}\bigg) \cdot \frac1{\sqrt \pi} \int_0^\infty e^{-(x-\eps^{-1/2}\ld c_k /2)^2} dx,
\end{split}
\end{align}

\noi 
which gives \eqref{Leps} as
\[
\PP (X_k\geq 0) = \frac1{\sqrt \pi} \int_0^\infty e^{-(x-\eps^{-1/2}\ld c_k /2)^2} dx,
\]

\noi 
with $X_k$ being independent normal random variables with mean $\eps^{-1/2} \ld c_k/2$ and variance $c_k^2/2$.
Thus, we finish  the proof of \eqref{Leps}.

Then, let us show the uniform boundedness of $\eps \Ld_\eps (\eps^{-1} \ld)$ for fixed $\ld >0$.
Without loss of generality, it only suffies to consider the case where $\ld = 1$.
We first note that 
\begin{align}\label{main}
\begin{split}
\eps \Lambda_\eps (\eps^{-1} ) 
& \le \eps \sum_{k \in \Z} \log \left( 1+ \sqrt{\pi}\varepsilon^{-1/2}\, c_k\, \exp\left(\frac{\varepsilon^{-1} \, c_k^2}{4}\right) \right).
\end{split}
\end{align}

\noi 
Then, we decompose the summation in \eqref{main} into two parts: when $k \in A$, we have $\varepsilon^{-1/2}\, c_k \le 1$; when $k \in B$, we have $\varepsilon^{-1/2}\, c_k > 1$. 
The summation in \eqref{main} over set $A$ can be bounded by
\begin{align}
\label{main1}
\begin{split}
\eps  \sum_{k \in A} \log \left( 1+ e^{1/4} \sqrt{\pi}\varepsilon^{-1/2}\, c_k\, \right) 
\le e^{1/4} \sqrt{\pi} \varepsilon^{1/2}\, \sum_{k \in A}  c_k\,,
\end{split}
\end{align}

\noi 
which is uniformly bounded as $(c_k) \in \l^1$, where we used $\log (1 + x) \le x$ for $x \ge 0$.

We turn to the summation in \eqref{main} over $k \in B$. 
We note that
\begin{align}
\label{main2}
\begin{split}
\eps & \sum_{k \in B} \log \left( 1+ \sqrt{\pi}\varepsilon^{-1/2}\, c_k\, \exp\left(\frac{\varepsilon^{-1} \, c_k^2}{4}\right) \right) \\
& \le \eps \sum_{k \in B} \log \left( \big( 1+ \sqrt{\pi}\varepsilon^{-1/2}\, c_k\, \big) \exp\left(\frac{\varepsilon^{-1} \, c_k^2}{4}\right) \right) \\
& \le \eps \sum_{k \in B} \log \big( 1+ \sqrt{\pi}\varepsilon^{-1/2}\, c_k\, \big)   + \eps \sum_{k \in B} \left(\frac{\varepsilon^{-1} \, c_k^2}{4}\right)  \\
& \le \sqrt{\pi}\varepsilon^{1/2} \sum_{k \in B}  \, c_k\,   + \frac14 \sum_{k \in B}  c_k^2,
\end{split}
\end{align}
which is again uniformly bounded as $(c_k) \in \l^1 \subset \l^2$, where we used $\eps^{1/2} \le c_k$ for $k \in B$.
\end{proof}

\begin{remark}
    \rm 
From \eqref{main}, \eqref{main1}, and \eqref{main2}, it follows that
\begin{align}
\label{main3}
\log \left( 1+ \sqrt{\pi}\varepsilon^{-1/2}\,\lambda\,c_k\, \exp\left(\frac{\varepsilon^{-1} \, \lambda^2 \, c_k^2}{4}\right) \, \PP (X_k\geq 0)\right) \le e^{1/4} \sqrt{\pi \ld} c_k + 2 \sqrt \pi \ld c_k^2,
\end{align}

\noi 
for all $\eps \in (0,1)$.
In particular, \eqref{main3} shows that the sequence $(a_k) = (e^{1/4} \sqrt{\pi \ld} c_k + 2 \sqrt \pi \ld c_k^2)$ is a dominate sequence of the series $\eps \Lambda(\eps^{-1} \ld)$ given in \eqref{Leps},
provided $(c_k) \in \l^1$.
\end{remark}

\medskip

With the above preparation, we would like to show
\begin{theorem}
With the above notation, we have
\begin{equation}
\Lambda(\lambda):= \lim_{\varepsilon \rightarrow 0^+} \varepsilon \Lambda_{\varepsilon}(\varepsilon^{-1}\lambda) = \frac{\lambda^2}{4}\, \sum_{j\in\Z} c_j^2,
\end{equation}

\noi 
if and only if $(c_k)\in \l^1 $.
\end{theorem}

\begin{proof}
We first consider the case when $(c_k) \in \l^1$.
From \eqref{Leps}, we rewrite
\begin{align}
    \label{up0}
\varepsilon \Lambda_{\varepsilon}(\varepsilon^{-1}\lambda) = \eps \sum_{k \in \Z} \Lambda_{\varepsilon}^k (\varepsilon^{-1} \, \lambda),
\end{align} 

\noi 
where
\[
\Lambda_{\varepsilon}^k (\varepsilon^{-1} \, \lambda) =  \log \left( 1+ \sqrt{\pi}\varepsilon^{-1/2}\,\lambda\,c_k\, \exp\left(\frac{\varepsilon^{-1} \, \lambda^2 \, c_k^2}{4}\right) \, \PP (X_k \geq 0)\right).
\]

\noi 
For fixed $k \in \Z$, it follows from L'Hospital's rule that
\begin{align}
\label{up1}
\begin{split}
\lim_{\varepsilon \rightarrow 0^+}\varepsilon \Lambda_{\varepsilon}^k (\varepsilon^{-1}\lambda) 
& = \lim_{\eps \to 0^+} \frac{\Lambda_{\varepsilon}^k(\varepsilon^{-1}\lambda) }{\frac1{\eps}} = \lim_{t \to \infty} 
\frac{\Lambda_{1/t}^k (t\lambda) }{t} \\
& = \lim_{t \to \infty} \frac{\log \left( 1+ \sqrt{\pi} t^{1/2}\,\lambda\,c_k\, \exp\left(\frac{t \, \lambda^2 \, c_k^2}{4}\right) \, \PP (X_k \geq 0)\right)}{t} \\
& =  \lim_{t \to \infty}  \frac{\frac12 \sqrt{\pi} t^{-1/2}\,\lambda\,c_k\, \exp\left(\frac{t \, \lambda^2 \, c_k^2}{4} \right) \, \PP (X_k \geq 0)}{ 1+ \sqrt{\pi} t^{1/2}\,\lambda\,c_k\, \exp\left(\frac{t \, \lambda^2 \, c_k^2}{4}\right) \, \PP (X_k \geq 0)} \\ 
&\hphantom{XXX}  + \lim_{t \to \infty}  \frac{ \frac14 \ld^2 c_k^2 \sqrt{\pi} t^{1/2}\,\lambda\,c_k\, \exp\left(\frac{t \, \lambda^2 \, c_k^2}{4}\right) \, \PP (X_k \geq 0) }{ 1+ \sqrt{\pi} t^{1/2}\,\lambda\,c_k\, \exp\left(\frac{t \, \lambda^2 \, c_k^2}{4}\right) \, \PP (X_k \geq 0)} \\ 
& = 0 +  \lim_{t \to \infty}  \frac{ \frac14 \ld^2 c_k^2 \sqrt{\pi} t^{1/2}\,\lambda\,c_k\, \exp\left(\frac{t \, \lambda^2 \, c_k^2}{4}\right) \, \PP (X_k \geq 0) }{ 1+ \sqrt{\pi} t^{1/2}\,\lambda\,c_k\, \exp\left(\frac{t \, \lambda^2 \, c_k^2}{4}\right) \, \PP (X_k \geq 0)} \\ 
& = \frac14 \ld^2 c_k^2.
\end{split}
\end{align} 


\noi 
We recall that, given $\ld > 0$, $\eps \Ld_\eps (\eps^{-1} \ld)$ is uniformly bounded in $\eps \in (0,1)$, provided $(c_k) \in \l^1$. 
In particular, from \eqref{main3} we see that the sequence $(a_k) = (e^{1/4} \sqrt{\pi \ld} c_k + 2 \sqrt \pi \ld c_k^2)$ is a control sequence of the series $\varepsilon \Lambda_{\varepsilon}(\varepsilon^{-1}\lambda)$ in the sense that $\eps \Lambda_{\varepsilon}^k (\varepsilon^{-1}\lambda) < a_k$ and $\sum_k a_k  < \infty$.
Then by Dominated Convergence Theorem, \eqref{up0}, and \eqref{up1}, it follows that
\[
\lim_{\eps \to 0^+} \varepsilon \Lambda_{\varepsilon}(\varepsilon^{-1}\lambda) = \sum_{ k\in\Z} \lim_{\eps \to 0^+} \varepsilon \Lambda_{\varepsilon}^k (\varepsilon^{-1}\lambda) = \frac{\ld^2}4 \sum_{ k\in\Z} c_k^2,
\]

\noi 
which finishes `if' part of the theorem.

Now we turn to the `only if' part of the theorem.
It suffices to show that 
\begin{align}
\label{onlyif}
\eps \Lambda (\eps^{-1} \ld) = \infty,
\end{align}

\noi 
for $(c_k)\notin \l^1$ and $\ld > 0$.
Without loss of generality, we only need to prove \eqref{onlyif} with $\ld = 1$.
To see this, we follow the same strategy as the proof of Lemma \ref{LEM:Leps}. 
From \eqref{main} and the fact $\PP (X_k \ge 0) \ge \frac12$, 
we see that
\begin{align}
\label{sharp}
\begin{split}
\eps \Lambda(\eps^{-1} ) & = \eps \sum_{k \in \Z} \log \left( 1+ \sqrt{\pi}\varepsilon^{-1/2}\, c_k\, \exp\left(\frac{\varepsilon^{-1} \,  c_k^2}{4}\right) \, \PP (X_k\geq 0)\right) \\
& \ge \eps \sum_{k \in \Z} \log \left( 1+ \frac12 \sqrt{\pi}\varepsilon^{-1/2}\, c_k\, \exp\left(\frac{\varepsilon^{-1} \, c_k^2}{4}\right) \right).
\end{split}
\end{align}

\noi 
We first consider the summation in the right-hand side of \eqref{sharp} over set 
$$B = \bigg\{k: \sqrt{\pi}\varepsilon^{-1/2}\, c_k\, \exp\left(\frac{\varepsilon^{-1} \, c_k^2}{4}\right) > 1\bigg\}.$$
If there are infinitely many elements in $B$,
then from \eqref{sharp} we have 
\begin{align}\label{sharp1}
\begin{split}
\eps \Lambda(\eps^{-1} )
& \ge \eps \sum_{k \in B} \log \left( 1+ \frac12 \sqrt{\pi}\varepsilon^{-1/2}\, c_k\, \exp\left(\frac{\varepsilon^{-1} \, c_k^2}{4}\right) \right)  \ge \sum_{k \in B} \log \left( 1+ \frac12 \right) = \infty.
\end{split}
\end{align}

\noi 
In what follows,
we assume the set $B$ is finite, 
which implies the set $A = \Z \setminus B$ is infinite.
Furthermore, since $(c_k) \in \l^1$, we see that 
$\sum_{k \in A} c_k = \infty$. 
Then, it follows that
\begin{align}\label{sharp3}
\begin{split}
\eps \Lambda(\eps^{-1} ) 
& \ge \eps \sum_{k \in A} \log \left( 1+ \sqrt{\pi}\varepsilon^{-1/2}\, c_k\, \exp\left(\frac{\varepsilon^{-1} \, c_k^2}{4}\right) \right) \\
& \ge \frac12 \sum_{ k \in A} \sqrt{\pi}\varepsilon^{1/2}\, c_k\, \exp\left(\frac{\varepsilon^{-1} \, c_k^2}{4}\right)\\
& \ge \frac12  \sqrt{\pi}\varepsilon^{1/2}\, \sum_{ k \in A} c_k = \infty, 
\end{split}
\end{align}

\noi 
where we used $\log (1+x) \ge \frac{x}2$ for $x \in (0,1)$ in the second inequality.
We finish the proof of \eqref{onlyif} and thus the theorem.
\end{proof}

\subsection{Linear LDP II - \texorpdfstring{$L^\infty$}{Lg} norm}
In this subsection,
we prove an LDP for the linear solution to \eqref{linear} under $L^\infty$ norm.
In particular, we shall prove
\begin{proposition}
\label{PROP:linear2}
Consider the linear Schr\"odinger equation on the torus $\T = [0,2\pi]$ as in \eqref{linear}, with random initial data $u_0^\o$ given by \eqref{random_data}.
Then,
\begin{align}
\label{LDP_linear2}
\lim_{\eps\to 0^+} \eps \log \PP \big( \|u(t)\|_{L^\infty (\T)} \ge z_0 \eps^{-1/2} \big) = - \frac{z_0^2}{\sum_{k \in \Z} c_k^2}, 
\end{align}

\noi 
provided $(c_k) \in \l^1$.
\end{proposition}

\begin{proof}
The lower bound of \eqref{LDP_linear2} has been proved in \cite[Subsection 2.1]{GGKS21}. 
We recall the proof here for completeness.
We first note the solution $u$ to the linear equation \eqref{linear} has the form
\begin{align}
\label{lin_sol}
u(t,x) = e^{-it \partial_{xx}} u_0 (x) = \sum_{k \in \Z} c_k g_k e^{ikx - ik^2 t},
\end{align}

\noi 
for $(t,x) \in \R \times \T$.
We note that the random field $u(t)$ is stationary Gaussian with mean $0$ and variance $\sum c_k^2$.
Thus, it follows from \eqref{RayDis} that
\begin{align*}
\begin{split}
\PP & \bigg( \sup_{x \in \T} | u(t,x)| \ge z_0 \eps^{-\frac12} \bigg) \ge \PP \big( | u(t,0)| \ge z_0 \eps^{-\frac12} \big) = \PP \big( | u(0,0)| \ge z_0 \eps^{-\frac12} \big) \\
& = \int_{z_0 \eps^{-1/2}} \frac{2x}{\sum c_k^2} \exp \bigg( - \frac{x^2}{\sum c_k^2} \bigg) dx = \exp \bigg( - \frac{z_0^2 \eps^{-1}}{\sum c_k^2} \bigg),
\end{split}
\end{align*}

\noi 
which in particular implies 
\begin{align}
\label{lower_b1}
\liminf_{\eps\to 0^+} \eps \log \PP \big( \|u(t)\|_{L^\infty (\T)} \ge z_0 \eps^{-1/2} \big) \ge - \frac{z_0^2}{\sum_{k \in \Z} c_k^2}.
\end{align}

The upper bound of \eqref{LDP_linear2} is a direct consequence of Proposition \ref{PROP:linear1}.
As a matter of fact, since $\|u(t)\|_{L^\infty (\T)} \le \|u(t)\|_{\F L^1 (\T)}$, we have
\[
\PP \bigg( \sup_{x \in \T} | u(t,x)| \ge z_0 \eps^{-\frac12} \bigg) \le \PP \big( \| u(t,x)\|_{\F L^1} \ge z_0 \eps^{-\frac12} \big).
\]

\noi 
Then from \eqref{LDP_linear1}, we have
\begin{align}
\label{upper_b}
\begin{split}
\limsup_{\eps\to 0^+} \eps \PP \bigg( \sup_{x \in \T} | u(t,x)| \ge z_0 \eps^{-\frac12} \bigg) & \le \limsup_{\eps\to 0^+} \eps \PP \big( \| u(t,x)\|_{\F L^1} \ge z_0 \eps^{-\frac12} \big)= - \frac{z_0^2}{\sum_{k \in \Z} c_k^2}.
\end{split}
\end{align}

Finally, by combining \eqref{lower_b1} and \eqref{upper_b}, we conclude \eqref{LDP_linear2}.
\end{proof}

\begin{remark}
\rm 
We note that Proposition \ref{PROP:linear2} improves \cite[Proposition 2.1]{GGKS21}, where the coefficients $(c_k)$ is required to decay exponentially, i.e. \eqref{decay}.
In Proposition \ref{PROP:linear2}, it only requires $(c_k) \in \l^1$.
\end{remark}


\section{LDP for the nonlinear problem}

In this section,
we upgrade the linear LDP in Proposition \ref{PROP:linear1} to its nonlinear version, i.e. Theorem \ref{THM:main}.
We prove Theorem \ref{THM:main} by following the argument in \cite[Section 4]{GGKS21}.
In particular, we use the resonant approximation of the equation \ref{weakNLS1}.
Due to the slow decay of the coefficients 4, we need to perturb the resonant system in a weaker solution space, specifically \(\FL^1 (\T)\), rather than the \(\FL^{2,1}\) space used in \cite{GGKS21}. 

\subsection{Resonant approximation}
In order to prove Theorem \ref{THM:main},
we introduce the resonant approximation.
Recall that for any continuous function $u(t,x)$, we have
\begin{align}
\label{FT}
u(t,x) = \sum_{k \in \Z} u_k (t) e^{ikx}.
\end{align}

\noi 
Then \eqref{weakNLS1} can be recast as
\begin{align}
\label{NLS1}
i \partial_t u_k - |k|^2 u_k = \eps^2 \sum_{k = k_1 - k_2+k_3} u_{k_1} \cj{u_{k_2}} u_{k_3}.
\end{align}

\noi 
Let $v (t,x) : = e^{it \partial_{xx}} u(t,x)$ be the interaction representation of $u(t,x)$.
Then we may rewrite \eqref{NLS1} as
\begin{align}
\label{NLS2}
i \partial_t v_k = \eps^2 \sum_{k = k_1 - k_2+k_3} v_{k_1} \cj{v_{k_2}} v_{k_3} e^{it \O},
\end{align}

\noi 
where
\begin{align}
\label{Omega}
\O = |k_1|^2 - |k_2|^2 + |k_3|^2 - |k|^2.
\end{align}

From the mass conservation of \eqref{NLS1} and \eqref{NLS2}, we see that the mass
\[
m = \| u(t)\|_{L^2(\T)}^2 = \| v(t)\|_{L^2(\T)}^2 = \| u(0)\|_{L^2(\T)}^2. 
\]

\noi
We define the gauge transform 
$$w (t,x) = e^{-2it\eps m} v(t,x).$$

\noi 
Under the gauge transform, \eqref{NLS2} can be recast as
\begin{align}
\label{NLS3}
i \partial_t w_k = \eps^2 \sum_{k = k_1 - k_2+k_3, \O \neq 0} w_{k_1} \cj{w_{k_2}} w_{k_3} e^{it \O}  - \eps^2 |w_k|^2 w_k .
\end{align}

The resonant approximation of the system \eqref{NLS3} is given by $a (t,x) = \sum_{k \in \Z} a_k (t) e^{ikx}$ such that
\begin{align}
\label{Rapp}
\begin{cases}
i \partial_t a_k = - \eps^2 |a_k|^2 a_k, \\
a_k (0) = c_k g_k.
\end{cases}
\end{align}

\noi 
The unique solution to \eqref{Rapp} is
\begin{align}
\label{a}
a(t,x) = \sum_{k \in \Z} c_k g_k e^{ikx + it \eps^2 c_k^2 |g_k|^2}.
\end{align}

\noi 
We also define
\begin{align}
\label{uapp}
u_{\rm app} (t,x) = e^{2it \eps^2 m} \sum_{k \in \Z} c_k g_k e^{ikx + it \eps^2 c_k^2 |g_k|^2 - it k^2}.
\end{align} 

\noi
As $(c_k) \in \l^1$, it is easy to see that $u_{\rm app}$ is well-defined globally in $t$ almost surely.
Furthermore, we remark that $a(t)$ and $u_{\rm app} (t)$ given in \eqref{a} and \eqref{uapp} are invariant under Fourier-Lebesgue norms, i.e. 
\begin{align}\label{invariant}
\| a(t) \|_{\FL^p} = \| a(0) \|_{\FL^p}, \quad \| u_{\rm app} (t) \|_{\FL^p} = \| u_{\rm app} (0) \|_{\FL^p} ,
\end{align}

\noi 
for all $p \in [1, \infty]$.

We also have the following LDP result with the Fourier-Lebesgue norm.

\begin{proposition}
\label{PROP:uapp2}
Given $z_0 > 0$ and $t > 0$,
we have that\begin{align}
\label{LDP_uapp2}
\lim_{\eps\to 0^+} \eps \log \PP \big(  \| u_{\rm app} (t,x) \|_{\FL^1 (\T)} 
\ge z_0 \eps^{-1/2} \big) = - \frac{ z_0^2}{\sum_{k \in \Z} c_k^2},
\end{align}

\noi 
if and only if $(c_k) \in \l^1$.
\end{proposition}

\begin{proof}
We note from \eqref{invariant} that
\[
\PP \big(  \| u_{\rm app} (t,x) \|_{\FL^1 (\T)} 
\ge z_0 \eps^{-1/2} \big) = \PP \big(  \| u (0) \|_{\FL^1 (\T)} 
\ge z_0 \eps^{-1/2} \big) ,
\]

\noi  
since $\|u_{\rm app} (t)\|_{\FL^1} = \|u_{\rm lin} (0)\|_{\FL^1}$, 
where $u_{\rm lin}$ is the linear solution to \eqref{linear}.
Then, \eqref{LDP_uapp2} follows from Proposition \ref{PROP:linear1}.
\end{proof}

\subsection{Bounds on the error}

We then have the following.

\begin{proposition}
\label{PROP:local}
If $u_0 \in \FL^1$, then there exists $T_\eps >0$ such that the unique solution to the IVP \eqref{weakNLS1} for all times $0 \le t \le T_\eps$. Moreover
\begin{align}
\label{Teps}
T_\eps \sim \eps^{-2} \|u_0\|_{\FL^1}^{-2}.
\end{align}

\noi 
In particular, if $\|v(t_1)\|_{\F L^1} < \infty$,
then we can extend the unique solution to \eqref{weakNLS1} from $t_1$ to $t_1 + \mathcal O (\eps^{-2} \|v(t_1)\|_{\F L^1}^{-2})$.
\end{proposition}

\begin{proof}
We first note that
\begin{align}
\label{vk}
v_k(t) = v_k(0) - i \eps^2 \int_0^t \sum_{k = k_1 - k_2 + k_3} (v_{k_1} \cj{v_{k_2}} v_{k_3} ) (s) e^{-is \O} ds
\end{align}

\noi 
and thus
\begin{align*}
\begin{split}
\|v (t)\|_{\FL^1} & \le \| v (0)\|_{\FL^1} + \eps^2 \sum_{k \in \Z} \int_0^t \Big| \sum_{k = k_1 - k_2 + k_3} (v_{k_1} \cj{v_{k_2}} v_{k_3} ) (s) \Big| ds \\
& \le \| v (0)\|_{\FL^1} + \eps^2 \sum_{k_1, k_2, k_3} \int_0^t \big|(v_{k_1} \cj{v_{k_2}} v_{k_3} ) (s) \big| ds \\
& \le \| v (0)\|_{\FL^1} + \eps^2 T  \|v (t)\|_{L^\infty_T \FL^1}^3.
\end{split}
\end{align*}

\noi 
The rest then follows.
\end{proof}

We have the following perturbation result:

\begin{proposition}
\label{PROP:pert}
Suppose that $u(t_1)$, $u_{\rm app} (t_1) \in \FL^1$,
and that $t_1 < t \le T_\eps$ as above. 
Then there exists some positive constant $C$ independent of $\eps$, $T_\eps$ and the initial data such that for all time $t_1 < t \le T_\eps$ the following inequality holds:
\begin{align}
\label{pert0}
\begin{split}
\|u(t) - & u_{\rm app} (t)\|_{\FL^{1}} \le \|u(t_1) - u_{\rm app} (t_1)\|_{\FL^{1}} \\
& + C \eps^2  (t-t_1)  \| u_{\rm app} (t_1)\|_{L^2(\T)}^2  \|u - u_{\rm app}\|_{L^\infty ([t_1,t), \FL^{1})}\\
& + C \eps^2 \|u (t_1)\|_{L^2(\T)}^2 \big( \|w(t) - a(t)\|_{\FL^1} + \|a (t_1) \|_{\F L^1} + \|w (t_1) \|_{\F L^1} \big) \\
& + C \eps^4 (t-t_1) \|u (t_1)\|^2_{L^2 (\T)} \Big( \|u - u_{\rm app}\|_{L^\infty ([t_1,t), \FL^{1})}^3 + \| u_{\rm app} (t_1)\|_{\FL^{1}}^3 \Big). 
\end{split}
\end{align}

\noi 
Similarly, we have
\begin{align}
\label{pert01}
\begin{split}
\|u(t) - & u_{\rm app} (t)\|_{L^\infty (\T)} \le \|u(t_1) - u_{\rm app} (t_1)\|_{L^\infty (\T)} \\
& + C \eps^2  (t-t_1)  \| u_{\rm app} (t_1)\|_{L^2(\T)}^2  \|u - u_{\rm app}\|_{L^\infty ([t_1,t), \FL^{1})}\\
& + C \eps^2 \|u (t_1)\|_{L^2(\T)}^2 \big( \|w(t) - a(t)\|_{\FL^1} + \|a (t_1) \|_{\F L^1} + \|w (t_1) \|_{\F L^1} \big) \\
& + C \eps^4 (t-t_1) \|u (t_1)\|^2_{L^2 (\T)} \Big( \|u - u_{\rm app}\|_{L^\infty ([t_1,t), \FL^{1})}^3 + \| u_{\rm app} (t_1)\|_{\FL^{1}}^3 \Big). 
\end{split}
\end{align}
\end{proposition}

\begin{proof}
We only prove \eqref{pert0}, as \eqref{pert01} follows from \eqref{pert0} since the embedding $\FL^1 \embeds L^\infty$.
Recall that $w_k$ is the solution to the gauged interaction form of NLS \eqref{NLS3} and $a_k$ is the solution to the interaction form of the resonant approximation \eqref{Rapp}.
We consider the difference
\begin{align}
\label{pert1}
i \partial_t (w_k - a_k) = \eps^2 (|a_k|^2 a_k - |w_k|^2 w_k) + \eps^2 \sum_{k=k_1 - k_2 + k_3, \O \neq 0}  w_{k_1} \cj{w_{k_2}} w_{k_3} e^{-it \O},
\end{align}

\noi 
where $\O$ is given by \eqref{Omega}. By integrating \eqref{pert1} from $t_1$ to $t$, we find that
\begin{align}
\label{pert2}
\begin{split}
i(w_k (t) - a_k (t)) = i & (w_k (t_1) - a_k (t_1)) + \eps^2 \int_{t_1}^t (|a_k|^2 a_k - |w_k|^2 w_k) (s) ds \\
& + \eps^2 \sum_{ \substack{k=k_1 - k_2 + k_3\\\O \neq 0}} \int_{t_1}^t ( w_{k_1} \cj{w_{k_2}} w_{k_3} ) (s) e^{-is \O} ds.
\end{split}
\end{align}

\noi 
In what follows, we shall estimate the right-hand side of \eqref{pert2}.
We first consider the second term on the right-hand side of \eqref{pert2}.
By the $L^2$ conservation of \eqref{NLS3} and \eqref{Rapp}, and H\"older inequality, we have that
\begin{align}
\label{pert3}
\begin{split}
\eps^2 & \sum_k \Big| \int_{t_1}^t (|a_k|^2 a_k - |w_k|^2 w_k) (s) ds \Big| \\
& \le 2 \eps^2 \int_{t_1}^t (\|a (s) \|_{\FL^2}^2 + \|w (s)\|_{\FL^2}^2) \| a (s) - w (s)\|_{\FL^\infty} ds \\
& \le 4 \eps^2  \|u (t_1)\|_{L^2(\T)}^2 \int_{t_1}^t \| a (s) - w (s)\|_{\FL^1} ds.
\end{split}
\end{align}

\noi 
which gives the second term of \eqref{pert0}.

Now, we turn to the third term on the right-hand side of \eqref{pert2}. 
By integrating by parts, 
we note that
\begin{align}
\label{pert4}
\begin{split}
\int_{t_1}^t ( w_{k_1} \cj{w_{k_2}} w_{k_3} ) (s) e^{-is \O} ds & = - \frac1{i \O} (w_{k_1} \cj{w_{k_2}} w_{k_3} ) (s) e^{-is \O} \Big|_{s={t_1}}^{s =t} \\
& \hphantom{X} + \frac1{i\O} \int_{t_1}^t \partial_s ( w_{k_1} \cj{w_{k_2}} w_{k_3} ) (s) e^{-is \O} ds.
\end{split}
\end{align}

\noi 
Recall the factorisation  $\O = 2(k_1 - k_2)(k_2 - k_3)$. Then, the boundary term of \eqref{pert4} can be estimated as follows:
\begin{align}
\label{pert5}
\begin{split}
\eps^2 & \sum_k \Bigg| \sum_{\substack{k = k_1 - k_2 + k_3\\ \O \neq 0}} \frac1{\O} \big[ (w_{k_1} \cj{w_{k_2}} w_{k_3} ) (t) e^{-it \O} - (w_{k_1} \cj{w_{k_2}} w_{k_3} ) ({t_1}) e^{-i t_1 \Omega} \big] \Bigg| \\
& \les \eps^2 \sum_{k_2 \neq k_1, k_3} \frac1{\jb{k_2 - k_1} \jb{k_2-k_3}} \bigg( \prod_{j=1}^3 |w_{k_j} (t)| +  \prod_{j=1}^3 |w_{k_j} (t_1)| \bigg) \\
& \les \eps^2 \big( \|w_k (t)\|_{\l^2}^2 \|w_k (t) \|_{\l^1} + \|w_k (t_1) \|_{\l^2}^2 \|w_k (t_1) \|_{\l^1} \big) \\
& \les \eps^2 \|u (t_1)\|_{L^2(\T)}^2 \big( \|w (t) \|_{\F L^1} + \|w (t_1) \|_{\F L^1} \big) \\
& \les \eps^2 \|u (t_1)\|_{L^2(\T)}^2 \big( \|w(t) - a(t)\|_{\FL^1} + \|a (t_1) \|_{\F L^1} + \|w (t_1) \|_{\F L^1} \big),
\end{split}
\end{align}

\noi 
which gives the third term of \eqref{pert0}.
In \eqref{pert5}, we used the $L^2$ conservation of $w(t)$ and $\F L^1$ conservation of $a(t)$.
Finally, we consider the second term in \eqref{pert4}. We first note from \eqref{NLS3} that
\begin{align}
\label{pert6}
\begin{split}
\|\partial_s w_k\|_{L^\infty ( [t_1,t); \FL^1)} \le 2 \eps^2 \|w_k\|_{L^\infty ( [t_1,t); \FL^1)}^3. 
\end{split}
\end{align}

\noi 
There are three cases depending on which term $\partial_s$ hits; we only consider one as an example.
\begin{align}
\label{pert7}
\begin{split}
\eps^2 & \sum_k \Big| \sum_{\substack{k = k_1 - k_2 + k_3, \\ \O \neq 0}} \frac1{\O} \int_{t_1}^t ( \partial_s w_{k_1}) \cj{w_{k_2}} w_{k_3}  e^{-i s \O} ds \Big| \\
& \les \eps^2 \int_{t_1}^t \sum_{k_2 \neq k_1,k_3} \frac1{\jb{k_2 - k_1} \jb{k_2-k_3}} | \partial_s w_{k_1} (s) {w_{k_2}} (s) w_{k_3} (s)| ds \\
& \les \eps^2 \int_{t_1}^t \|w (s)\|^2_{L^2 (\T)} \| \partial_s w (s)\|_{\F L^1} ds \\
& \les \eps^4 (t-t_1) \|u ({t_1})\|^2_{L^2 (\T)} \Big( \|u - u_{\rm app}\|_{L^\infty ([{t_1},t), \FL^{1})}^3 + \| u_{\rm app} ({t_1})\|_{\FL^{1}}^3 \Big),
\end{split}
\end{align}

\noi 
which gives the fourth term in \eqref{pert0} and thus finish the proof.
\end{proof}

As a consequence of the above proposition, we have the following.

\begin{proposition} 
\label{PROP:p}
 Fix $\dl \in (0,1)$, $n \in \mathbb N$, and $d_1, d_2 > 0$.
 Then there exists $\eps_0$ depending on $\dl, n, d_1$ and $d_2$,
 and $\ld >0$ depending on $d_2$,
 such that the following holds for all $\eps \le \eps_0$.
 If $T = \mathcal O ( \eps^{-1} |\log \eps|)$ and $\|u_0\|_{\FL^{1}} \le d_2 \eps^{-1/2}$, then
 \begin{itemize}
     \item[(i)\,\,] we can extend the time where 
     \eqref{pert0} is valid from $T_\eps \sim \eps^{-2} \|u_0\|_{\FL^{1}}^{-2}$ all the way to $T$; and
     \item[(ii)] we have the estimate 
     \begin{align}
        \label{p2}
         \|u - u_{\rm app}\|_{L^\infty ([0,T], \FL^{1})} < \eps^{-1/2 + \dl}.
     \end{align}
 \end{itemize} 
\end{proposition}

\begin{proof}
We first note that \eqref{pert0} is valid for any time $t$ such that $u(t) \in \FL^1$, which is guaranteed for all $t$ such that $t \le T_\eps$.
We also note that $T_\eps \sim \eps^{-2}$. 
We shall extend the existence time from $T_\eps$ to $T = n \ld \eps^{-1}$.

We start by setting $u(0) = u_{\rm app} (0)$.
Then we define
\[
\tau_1 : = \sup \{ t \in [0,T]: \| u(t) - u_{\rm app} (t) \|_{\FL^1} < \eps^{-1/2 + 2\dl}\},
\]

\noi 
for some small $\dl \in (0,\frac12)$.
If $t \le \tau_1$ we then have
\[
\begin{split}
\|u(t) \|_{\FL^1} & \le \|u(t) - u_{\rm app} (t) \|_{\FL^1} + \|u_{\rm app} (t) \|_{\FL^1} \\
&  \le \|u(t) - u_{\rm app} (t) \|_{\FL^1} + \|u_{\rm app} (0) \|_{\FL^1}  \\
&  \le \eps^{-1/2 + 2\dl} + d_2 \eps^{-1/2}, 
\end{split}
\]

\noi 
and thus $u (t) \in \FL^1$ for $t \le \tau_1$,
which implies that \eqref{pert0} is are valid till $\tau_1$.
Moreover, $u(t)$ must be continuous in time to live in $\FL^1$ thanks to the local well-posedness.

If $t \le \min \{\tau_1, \ld \eps^{-1}\}$ for some $\ld > 0$ to be determined later,
then from \eqref{pert0} with $t_1 = 0$ we have
\[
\begin{split}
\| u(t) - u_{\rm app} (t) \|_{\FL^1} & \le C d_2^2 \eps^3 t (\eps^{-3/2 + 6\dl} + d_2^3 \eps^{-3/2}) \\
& \hphantom{X} + C d_2^2 \eps t \|u - u_{\rm app}\|_{L^\infty ([0,t), \FL^{\infty})} \\
& \hphantom{X} + C \eps^2 (\eps^{-3/2 + 6\dl} + 2d_2^3 \eps^{-3/2}) \\ 
 & \le C \ld \eps^{\frac12} (\eps^{6\dl} + d_2^3)  + C \eps^{\frac12} (\eps^{6\dl} + 2d_2^3 ) + C \ld d_2^2 \eps^{-\frac12 + 2\dl}  \\
& \le \frac14 \eps^{-1/2 + 2\dl},
\end{split}
\]

\noi 
provided $\eps \ll 1$ and $\ld \le \frac1{8C d_2^2}$.
This together with the definition of $\tau_1$ implies that $\tau_1 \ge \ld \eps^{-1}$.
Therefore, we have shown that
\begin{align}
\label{p3}
\| u(t) - u_{\rm app} (t) \|_{\FL^1} < \eps^{-1/2 + 2\dl}
\end{align}

\noi 
holds for all $0 \le t \le \ld \eps^{-1}$, provided $\eps \gg 1$.

Let us fix $t_1 : = \ld \eps^{-1}$,
and restart our equation for time $t > t_1$.
We would like o show that we may reach times $t_2 = 2 \ld \eps^{-1}$ while the error is still not too big.
To do this, let 
\[
\tau_2 : = \sup \{ t \in [t_1,T]: \| u(t) - u_{\rm app} (t) \|_{\FL^1} < 2 \eps^{-1/2 + 2 \dl}\}.
\]

\noi 
First, we note that for $t \le \tau_2$ we have
\begin{align}
\label{p4}
\begin{split}
\|u(t) \|_{\FL^1} & \le \|u(t) - u_{\rm app} (t) \|_{\FL^1} + \|u_{\rm app} (t_1) \|_{\FL^1} \\
&  \le 2 \eps^{-1/2 + 2\dl}  + \|u_{\rm app} (0) \|_{\FL^1}  \\
&  \le 2 \eps^{-1/2 + 2\dl} + d_2 \eps^{-1/2}, 
\end{split}
\end{align}

\noi 
which guarantees the existence of $u(t)$ after $t_1$ and up to $\tau_2$ by Proposition \ref{PROP:local}.
Then, we consider $t_1 \le t \le \min \{\tau_2, t_1 + \ld \eps^{-1}\}$ for some $\ld_2 > 0$ to be fixed later.
We then use \eqref{pert0} with initial data at $t_1$ instead of $0$ to get
\begin{align} 
\label{iter1}
\begin{split}
\|u(t) - & u_{\rm app} (t)\|_{\FL^{1}} 
\le \|u(t_1) - u_{\rm app} (t_1)\|_{\FL^{1}} \\
& \hphantom{X} + C \eps^4 (t-t_1) \| u_{\rm app} (t_1)\|_{L^2(\T)}^2 \Big( \|u - u_{\rm app}\|_{L^\infty ([t_1,t), \FL^{1})}^3 + \| u_{\rm app} (t_1)\|_{\FL^{1}}^3 \Big)\\
&  \hphantom{X} + C \eps^2 (t-t_1) \| u_{\rm app} (t_1)\|_{L^2(\T)}^2 \|u - u_{\rm app}\|_{L^\infty ([t_1,t), \FL^{1})}\\
&  \hphantom{X} + C \eps^2 \| u_{\rm app} (t_1)\|_{L^2(\T)}^2 \Big( \|u - u_{\rm app}\|_{L^\infty ([t_1,t), \FL^{1})} + \| u_{\rm app} (t_1)\|_{\FL^{1}} + \|u(t_1)\|_{\FL^{1}} \Big)\\
& \le \eps^{-1/2 + 2\dl} + C \eps^{\frac12} \ld (2^3 \eps^{6\dl} + d_2^3 )  + C \ld  d_2^2 \|u - u_{\rm app}\|_{L^\infty ([t_1,t), \FL^{1})} \\
& \hphantom{X}  + C d_2^2 \eps^{\frac12} (2^3 \eps^{ 2\dl} + d_2  + (\eps^{ 2\dl} + d_2) ) \\
& \le \eps^{-1/2 + 2\dl} +  \frac14 \eps^{-1/2 + 2\dl} + C \ld  d_2^2 \|u - u_{\rm app}\|_{L^\infty ([t_1,t), \FL^{1})},
\end{split}
\end{align}

\noi 
by choosing $\eps \ll 1$ even smaller. 
Recall that $C \ld d_2^2 < \frac18$.
Then we have from \eqref{iter1} that
\begin{align}
\label{iter2}
\|u - u_{\rm app}\|_{L^\infty ([t_1,t], \FL^{1})} \le \frac{10}7 \eps^{-1/2 + 2\dl} < 2 \eps^{-1/2 + 2\dl},
\end{align}

\noi 
holds for all $t_1 \le t \le \min\{\tau_2, t_1 + \ld \eps^{-1}\}$.
In particular, the above argument shows that $\tau_2 \ge t_1 + \ld \eps^{-1} = 2 \ld \eps^{-1}$, i.e. we have that 
\begin{align}
\label{p5}
\|u(t) - u_{\rm app} (t)\|_{\FL^{1}} 
< 2 \eps^{-1/2 + 2\dl},
\end{align}

\noi 
holds for all $0\le t \le 2\ld \eps^{-1}$.

By repeating the above procedure and choosing $t_n = n \ld \eps^{-1}$ and
\[
\tau_n : = \sup \{ t \in [t_{n-1},T]: \| u(t) - u_{\rm app} (t) \|_{\FL^1} < 2^{n-1} \eps^{-1/2 + 2\dl}\},
\]

\noi 
similar argument as in \eqref{iter1} yields
\begin{align}
\label{iter11}
\|u(t) - u_{\rm app} (t)\|_{\FL^{1}} 
\le 2^{n-2} \eps^{-1/2 + 2\dl} +  \frac14 \eps^{-1/2 + 2\dl} + C \ld  d_2^2 \|u - u_{\rm app}\|_{L^\infty ([t_1,t), \FL^{1})}
\end{align}

\noi 
for $t \in [t_{n-1}, \min\{\tau_n, t_{n-1} + \ld \eps^{-1}) $. 
Recall that $C \ld d_2^2 < \frac18$.
Then we have
\begin{align}
\label{iter3}
\|u - u_{\rm app}\|_{L^\infty ([t_{n-1}, t], \FL^{1})} \le \frac{2 \cdot 2^{n} + 2}7 \eps^{-1/2 + \dl} < 2^n \eps^{-1/2 + 2\dl},
\end{align}

\noi 
holds for all $t_{n-1} \le t \le \min\{\tau_n, t_{n-1} + \ld \eps^{-1}\}$.
In particular, the above argument shows that $\tau_n \ge t_{n-1} + \ld \eps^{-1} = n \ld \eps^{-1}$, i.e. we have that 
\[
\|u(t) - u_{\rm app} (t)\|_{\FL^{1}} 
< 2^{n-1} \eps^{-1/2 + 2\dl},
\]

\noi 
holds for all $0\le t \le n\ld \eps^{-1}$. 

Finally, we shall have
\begin{align}
\label{p51}
\|u(t) - u_{\rm app} (t)\|_{\FL^{1}} 
< \eps^{-1/2 + \dl},
\end{align}

\noi 
provided $2^{n-1} \eps^\dl \le 1$, i.e. $n = \mathcal O( |\log \eps| )$. Therefore, the time scale that \eqref{p5} holds is $t = \mathcal O ( n\ld \eps^{-1}) = \mathcal O ( \eps^{-1} |\log \eps|) $.
Thus, we finish the proof.
\end{proof}

\subsection{Proof for the main theorem}
\label{SUB:mainproof}
For fixed $z_0 > 0$ (independent of $\eps$), $\dl \in (0,1)$, and $n \in \mathbb N$, 
let $t \le n \ld \eps^{-1}$ for some fixed $\ld > 0$ depends on $z_0$.
We want to study the limit
\begin{align}
\label{LDPc}
\lim_{\eps \to 0^+} \eps \log \mathbb P \big( \|u(t,x) \|_{\FL^1 (\T)} > z_0 \eps^{-1/2} \big).
\end{align}

\noi 
Note that the solution $u$ exists at time $t$ and $\| u(t)\|_{\F L^1 (\T)} < \infty$ almost surely by choosing $\eps $ sufficiently small.
Let $\D_\eps$ be the event 
\[
\D_\eps : = \big\{\o \in \O: \|u(t,x) \|_{\FL^1 (\T)}  > z_0 \eps^{-1/2} \big\}. 
\]

We first prove the upper bound for $\mathbb P (\mathcal D_\eps)$.
It is clear that $\D_\eps \subset \D_\eps^+$ where
\[
\D_\eps^+ : = \big\{ \| u_{\rm app}(t)\|_{\FL^1 (\T)} + \| u(t) - u_{\rm app} (t) \|_{\F\!L^{1}} > z_0 \eps^{-1/2} \big\}.
\]

\noi 
Let us also define the event 
\begin{align}\label{Beps}
\B_\eps : = \{\| u(t) - u_{\rm app} (t) \|_{\F\!L^{1}} \le z_0 \eps^{-1/2 + \dl}\} .
\end{align}

\noi 
Then we have that 
\begin{align}
\label{upper}
\mathbb P (\D_\eps) \le \mathbb P (\D_\eps^+ \cap \B_\eps) + \mathbb P (\D_\eps^+ \cap \B_\eps^c) .
\end{align}

\noi 
For the first term in \eqref{upper}, we have
\[
\mathbb P (\D_\eps^+ \cap \B_\eps) \le \PP \big( \| u_{\rm app} (t) \|_{\FL^1 (\T)} > z_0 (\eps^{-1/2} - \eps^{-1/2 + \dl}) \big), 
\]

\noi 
From Proposition \ref{PROP:uapp2}, it follows
\begin{align}
\label{upper1}
\lim_{\eps \to 0^+} \eps \log \PP  \big( \| u_{\rm app} (t) \|_{\FL^1 (\T)} > z_0 (\eps^{-1/2} - \eps^{-1/2 + \dl}) \big) = - \frac{z_0^2}{\sum_{k \in \Z} c_k^2}.
\end{align}

\noi
Next, we consider the second term in \eqref{upper}.
Fix $d_2 > 0$ large enough so that
\begin{align}
\label{upper2}
d_2^2 > 2 z_0^2.
\end{align}

\noi
With Proposition \ref{PROP:p}, if $\eps < \eps_0$ is small enough we can arrange
\[
\mathbb P (\D_\eps^+ \cap \B_\eps^c) \le \mathbb P ( \B_\eps^c ) \le \mathbb P (\|u_0\|_{\FL^{1}} \ge d_2 \eps^{-1/2}).
\]

\noi 
where the second inequality we used Proposition \ref{PROP:p}.
We note that 
\begin{align}
\label{pp3}
\begin{split}
\lim_{\eps \to 0^+} \eps \log \PP (\B_\eps^c) \le \lim_{\eps \to 0^+} \eps \log \PP   &\big( \| u_{0} \|_{\FL^{1}} > d_2 \eps^{-1/2} \big) \le  - \frac{d_2^2}{ \sum_{k \in \Z} c_k^2} \le  - \frac{2z_0^2}{\sum_{k \in \Z} c_k^2}.
\end{split}
\end{align}

\noi 
Thanks to the choice of $d_2$, \eqref{pp3} implies that \eqref{upper1} is asymptotically larger than that of \eqref{pp3}.
In particular, we have
\[
\PP (\mathcal D_\eps) \le 3 \PP  \big( \| u_{\rm app} (t) \|_{\FL^1 (\T)} > z_0 (\eps^{-1/2} -   \eps^{-1/2 + \dl}) \big)
\]

\noi 
for $\eps$ small enough, which concludes the proof of the upper bound due to \eqref{upper1}.

Next, let us consider the lower bound.
First we define
\[
\D_\eps^- : = \big\{ \| u_{\rm app}(t)\|_{\FL^1 (\T)} - \| u(t) - u_{\rm app} (t) \|_{\F\!L^{1}} > z_0 \eps^{-1/2} \big\}.
\]

\noi 
Then we have that
\begin{align}
\label{pp4}
\begin{split}
\PP (\D_\eps) & \ge \PP (\D_\eps^- \cap \B_\eps) \\
& \ge \PP \Big( \big\{ \| u_{\rm app}(t)\|_{\FL^1 (\T)} > z_0 (\eps^{-1/2} + \eps^{-1/2 + \dl} )\big\} \cap \B_\eps \Big) \\
& \ge \PP \Big( \big\{ \| u_{\rm app}(t)\|_{\FL^1 (\T)} > z_0 (\eps^{-1/2} +  \eps^{-1/2 + \dl} )\big\} \Big) - \PP ( \B_\eps^c ).
\end{split}
\end{align}

\noi 
We first note that 
\begin{align}
\label{pp40}
\lim_{\eps \to 0^+} \eps \log \PP  \big( \| u_{\rm app} (t) \|_{\FL^1 (\T)} > z_0 (\eps^{-1/2} + \eps^{-1/2 + \dl}) \big) = - \frac{z_0^2}{\sum_{k \in \Z} c_k^2}.
\end{align}

\noi 
Then from \eqref{pp3}, it shows that $\PP (\B_\eps^c)$ is asymptotically smaller than $\frac 12 \PP  \big( \| u_{\rm app} (t) \|_{\FL^1 (\T)} > z_0 (\eps^{-1/2} + \eps^{-1/2 + \dl}) \big)$. In particular, we can deduce that
\begin{align}
\label{pp6}
\PP (\D_\eps) \ge \frac12 \PP  \big( \| u_{\rm app} (t) \|_{\FL^1 (\T)} > z_0 (\eps^{-1/2} + \eps^{-1/2 + \dl}) \big)
\end{align}

\noi 
for $\eps$ small enough.
Therefore, we conclude the proof by using \eqref{pp40} and \eqref{pp6}.
Furthermore, we note that \eqref{pp40} and \eqref{pp6} together with Proposition \ref{PROP:uapp2} imply the sharpness of the decay condition $(c_k) \in \l^1$ in Theorem \ref{THM:main}.

\subsection{Proof of Corollary \ref{COR:main}}
In this subsection, 
we sketch the proof of Corollary \ref{COR:main}.
Let $u$ be the solution to \eqref{weakNLS1} with initial data $u_0$ in \eqref{random_data}.
Since $\|u\|_{L^\infty (\T)} \le \| u\|_{\FL^1 (\T)}$, we note that
\[
\mathbb P \big( \|u(t,x) \|_{\FL^1 (\T)} > z_0 \eps^{-1/2} \big) \ge \mathbb P \bigg( \sup_{x \in \T} |u(t,x) | > z_0 \eps^{-1/2} \bigg),
\]

\noi 
which implies that 
\begin{align}\label{upperbound}
\limsup_{\eps \to 0^+} \eps \log \PP \Big( \sup_{x} |u(t,x) | > z_0\eps^{-1/2} \Big) \le \lim_{\eps \to 0^+} \eps \log \PP \Big( \|u(t,x) \|_{\FL^1} > z_0\eps^{-1/2} \Big),
\end{align}

\noi 
provided $(c_k) \in \l^1$.

It then suffices to show the lower bound
\begin{align}\label{lowerbound}
\liminf_{\eps \to 0^+} \eps \log \PP \Big( \sup_{x} |u(t,x) | > z_0\eps^{-1/2} \Big) \ge \lim_{\eps \to 0^+} \eps \log \PP \Big( \|u(t,x) \|_{\FL^1} > z_0\eps^{-1/2} \Big).
\end{align}

\noi 
First, we define 
\[
\wt \D_\eps : = \big\{\o \in \O: \|u(t,x) \|_{L^\infty (\T)}  > z_0 \eps^{-1/2} \big\},
\]
and
\[
\wt \D_\eps^- : = \big\{ \| u_{\rm app}(t)\|_{L^\infty (\T)} - \| u(t) - u_{\rm app} (t) \|_{\FL^1 (\T)} > z_0 \eps^{-1/2} \big\}.
\]

\noi 
Then we have that
\begin{align*} 
\begin{split}
\PP (\wt \D_\eps) & \ge \PP (\wt \D_\eps^- \cap \B_\eps) \\
& \ge \PP \Big( \big\{ \| u_{\rm app}(t)\|_{L^\infty (\T)} > z_0 (\eps^{-1/2} + \eps^{-1/2 + \dl} )\big\} \cap \B_\eps \Big) \\
& \ge \PP \Big( \big\{ \| u_{\rm app}(t)\|_{L^\infty (\T)} > z_0 (\eps^{-1/2} +  \eps^{-1/2 + \dl} )\big\} \Big) - \PP ( \B_\eps^c ),
\end{split}
\end{align*}

\noi 
where $\mathcal B_\eps$ is defined in \eqref{Beps}.
We first note, see also \cite[Proposition 4.1]{GGKS21},\footnote{Please note that \cite[Proposition 4.1]{GGKS21} only requires that $\|c_k\|_{\l^1_k} < \infty$.} that 
\begin{align}
\label{pp400}
\lim_{\eps \to 0^+} \eps \log \PP  \big( \| u_{\rm app} (t) \|_{L^\infty (\T)} > z_0 (\eps^{-1/2} +  \eps^{-1/2 + \dl}) \big) = - \frac{z_0^2}{\sum_{k \in \Z} c_k^2}.
\end{align}

\noi 
Then from \eqref{pp3}, it shows that $\PP (\B_\eps^c)$ is asymptotically smaller than $\frac 12 \PP  \big( \| u_{\rm app} (t) \|_{L^\infty (\T)} > z_0 (\eps^{-1/2} + \eps^{-1/2 + \dl}) \big)$. In particular, we can deduce that
\begin{align}
\label{pp60}
\PP (\wt \D_\eps) \ge \frac12 \PP  \big( \| u_{\rm app} (t) \|_{L^\infty (\T)} > z_0 (\eps^{-1/2} + \eps^{-1/2 + \dl}) \big)
\end{align}

\noi 
for $\eps$ small enough.
Therefore, we conclude the proof of \eqref{lowerbound} by using \eqref{pp400} and \eqref{pp60}.

\subsection*{Acknowledgements}

R.L. would like to thank Andrea R. Nahmod for her support and Ricardo Grande for interesting discussions. Y.W. is supported by the EPSRC New Investigator Award (grant no. EP/V003178/1).


\begin{thebibliography}{99}

\bibitem{BCD}
H.~Bahouri, J.-Y.~Chemin, 
R.~Danchin, 
{\it Fourier analysis and nonlinear partial differential equations,}
Grundlehren der Mathematischen Wissenschaften [Fundamental Principles of Mathematical Sciences], 
343. Springer, Heidelberg, 2011. xvi+523 pp.

\bibitem{BGMS25}
M.~Berti, R.~Grande, A.~Maspero, G.~Staffilani,
{\it Rogue waves and large deviations for 2D pure gravity deep water waves},
arXiv:2510.15159.

\bibitem{BO93}
J.~Bourgain, {\it Fourier transform restriction phenomena for certain lattice subsets and applications to nonlinear evolution equations. I. Schr\"odinger equations}, Geom. Funct. Anal. 3 (1993), 107--156.

\bibitem{BO94}
J.~Bourgain, {\it Periodic nonlinear Schr\"odinger equation and invariant measures}, 
Comm. Math. Phys. 166 (1) (1994), 1--26.

\bibitem{BO96}
J.~Bourgain, 
{\it Invariant measures for the 2D-defocusing nonlinear Schr\"odinger equation}, 
Comm. Math. Phys. 176 (2) (1996), 421--445.

\bibitem{DGV18}
G.~Dematteis, T.~Grafke, E.~Vanden-Eijnden,
{\it Rogue waves and large deviations in deep sea},
Proc. Natl. Acad. Sci. USA 115 (5) (2018), 855--860.

\bibitem{DGV19}
G.~Dematteis, T.~Grafke, E.~Vanden-Eijnden,
{\it Extreme event quantification in dynamical systems with random components},
SIAM/ASA J. Uncertain. Quantif. 7 (3) (2019), 1029--1059.

\bibitem{DS89}
J.~Deuschel, D. W.~Stroock,
{\it Large deviations},
Pure Appl. Math., 137,
Academic Press, Inc., Boston, MA, 1989, xiv+307 pp.

\bibitem{GGKS21}
M.A.~Garrido, R.~Grande, K.M.~Kurianski, G.~Staffilani, 
{\it Large deviations principle for the cubic NLS equation}, 
Commun. Pure Appl. Math. 76 (2023), 4087–4136.

\bibitem{SS99}
C.~Sulem, P.L.~Sulem, 
{\it The Nonlinear Schr\"odinger Equations: Self-Focusing and Wave Collapse}, 
Applied Mathe-matical Sciences, vol.139, Springer-Verlag, New York, 1999, 350pp.

\end{thebibliography}
\end{document}